\documentclass[11pt, reqno]{amsart}
\usepackage[
    textwidth=14.8cm, 
    textheight=22.5cm, 
    heightrounded, 
    centering
]{geometry}

\usepackage{amsfonts,amssymb, amscd}
\usepackage{amsmath}
\usepackage{lmodern}
\usepackage{mathrsfs}
\usepackage{amsthm}
\usepackage{qsymbols}
\usepackage{mathtools}
\mathtoolsset{showonlyrefs}
\usepackage{dsfont}
\usepackage{graphicx}
\usepackage{latexsym}
\usepackage{chngcntr}
\usepackage[noadjust]{cite}
\usepackage{paralist}
\usepackage[parfill]{parskip}
\usepackage{bm}
\usepackage{tikz-cd}
\usepackage{nicefrac}
\usepackage{float}
\usepackage{csquotes}

\usepackage[shortlabels]{enumitem}
\setlist[enumerate,1]{label = \normalfont(\roman*), ref = (\roman*)}
\usepackage{euflag}

\usepackage{todonotes}
\makeatletter
\@mparswitchfalse%
\makeatother
\normalmarginpar %

\newtheorem{theorem}{Theorem}[section]
\newtheorem{lemma}[theorem]{Lemma}

\newtheorem{proposition}[theorem]{Proposition}

\newtheorem{taggedtheoremx}{Assumption}
\newenvironment{assumption}[1]
{\renewcommand\thetaggedtheoremx{#1}\taggedtheoremx}
{\endtaggedtheoremx}
\theoremstyle{definition}
\newtheorem{definition}[theorem]{Definition}
\newtheorem{remark}[theorem]{Remark}
\newtheorem{example}[theorem]{Example}

\usepackage[plainpages=false,pdfpagelabels,
citecolor=red]{hyperref}
\usepackage{xcolor}
\hypersetup{
	colorlinks,
	linkcolor={cyan!90!black},
	citecolor={magenta},
	urlcolor={green!40!black}
} %

\usepackage{doi}

\newcommand{\F}{\mathcal{F}}
\renewcommand{\P}{\mathbb{P}}
\newcommand{\E}{\mathbb{E}}
\newcommand{\R}{\mathbb{R}}

\newcommand{\N}{\mathbb{N}}

\newcommand{\Cont}{\mathrm{C}}

\newcommand{\e}{\mathrm{e}}
\renewcommand{\d}{\,\mathrm{d}}
\newcommand{\Dom}{D}

\newcommand{\eps}{\varepsilon}

\newcommand{\Om}{\Omega}
\newcommand{\ceqq}{\coloneqq}
\newcommand{\col}{\colon}

\newcommand{\I}{\textrm{I}}
\newcommand{\II}{\textrm{II}}
\newcommand{\III}{\textrm{III}}
\newcommand{\IV}{\textrm{IV}}
\newcommand{\VV}{\textrm{V}}

\makeatletter
\g@addto@macro\bfseries{\boldmath}
\makeatother

\makeatletter
\@namedef{subjclassname@2020}{%
	\textup{2020} Mathematics Subject Classification}
\makeatother

\title[Strong existence and uniqueness for quasilinear SPDEs]{Strong existence and uniqueness for a class of quasilinear stochastic evolution equations} 
\author{Sebastian Bechtel}
\address{Université Paris-Saclay, CNRS \\ Laboratoire de Mathématiques d’Orsay \\ 91405 Orsay \\ France}
\email{sebastian.bechtel@universite-paris-saclay.fr}
\author{Esm\'ee Theewis}
\address{Delft Institute of Applied Mathematics, Delft University of Technology, 2600 GA Delft, The Netherlands}
\email{e.s.theewis@tudelft.nl}
\date{July 14, 2026}
\thanks{This project has received funding from the European Union’s Horizon 2020 research and innovation programme under the Marie Skłodowska-Curie grant agreement No 101034255 \euflag{}.  The second author is supported by the VICI subsidy VI.C.212.027 of the Netherlands Organisation for Scientific Research (NWO)}
\dedicatory{}
\begin{document}
\begin{abstract}
We establish existence of probabilistically strong solutions and pathwise uniqueness for a class of quasilinear stochastic evolution equations on bounded domains.
Our results combine recent weak existence results for quasilinear stochastic evolution equations in an $L^p$-setting (with $p > 2$) with Yamada--Watanabe theory.
To establish pathwise uniqueness, we rely on an $L^1$-contraction argument. 
\end{abstract}

\maketitle

\section{Introduction}

In this article, we investigate strong existence and pathwise uniqueness of global solutions to the quasilinear stochastic evolution equation on a bounded, open set $D\subset\R^d$: 
	\begin{align}
		\label{eq:ql}
		\tag{QL}
		\left\{\quad
		\begin{aligned}
			\d u  &= \big[ \partial_i (a_{ij}(u) \partial_j u) + \partial_i \Phi_i(u) +\phi(u) \big] \d t
			\\  & \qquad \qquad
			+ \sum_{n\geq 1}  \big[ b_{n,j}(u) \partial_j u + g_{n}(u) \big] \d w_n,\ \ &\text{ on }\Dom,\\
			u &= 0, \ \ &\text{ on }\partial\Dom,
			\\
			u(0)&=u_{0},\ \ &\text{ on }\Dom,
		\end{aligned}\right.
	\end{align}
where $u$ is scalar-valued and we use Einstein's summation convention. 
A common strategy to establish such results involves three steps: proving the existence of a probabilistically weak solution by means of a stochastic compactness method, establishing pathwise uniqueness  via an $L^1$-contraction argument, and finally applying Yamada–Watanabe-type results to deduce strong existence. While this strategy has been successfully applied to quasilinear equations on the torus in~\cite{DHV,HZ} and  conservative SPDEs on the full space in~\cite{FG}, equations on domains with boundary conditions were not treated yet.

As a first step towards equations on domains, existence of probabilistically weak solutions to \eqref{eq:ql} was established for finite time intervals $[0,T]$ in the article~\cite{BV} by the first-named author and Veraar -- even in the case of systems. The authors raise the question whether strong existence and pathwise uniqueness results could also be proved in their setting (see \cite[Rem.\ 6.8]{BV}). Using the approach outlined above, we answer their question in the affirmative in this work.

From now on, we fix a filtered probability space $(\Omega,\F,\P,(\F_t))$, an $\ell^2$-cylindrical Brownian motion $W=(w_n)$ on $(\Omega,\F,\P,(\F_t))$ and a $\F_0$-measurable initial datum  $u_0\colon \Omega\to L^2(D)$. 

The coefficients $a$ and $b$ and the nonlinearities $(\Phi_i)$, $\phi$ and $g$ in~\eqref{eq:ql} will be assumed to satisfy certain coercivity, dissipativity, growth and (local) Lipschitz assumptions, which will be specified in Assumption~\ref{ass:ql} for some $h > 1$ describing the order of growth.
Moreover, the initial datum $u_0$ is assumed to take values in $B_{2,p,0}^{1-\nicefrac{2}{p}}(D) \cap L^q(D)$ and satisfies appropriate moment conditions. %
The subscript zero for the Besov space indicates compatibility of the initial datum with the Dirichlet boundary conditions in~\eqref{eq:ql}. 
The parameters $p$ and $q$ are as follows: the ellipticity constants in Assumption~\ref{ass:ql} determine some $p_0\in(2,\infty)$. Then, exponents $p\in (2, p_0]$ and $q\in (ph, \infty)$ can be chosen. The condition $p > 2$ originates from the weak existence result in~\cite{BV}. It is leveraged to obtain powerful tightness results proving useful in the stochastic compactness method. 

Regarding the application of Yamada--Watanabe results, we will rely on the recent article \cite{T} by the second-named author, which fully covers our non-variational setting with time-dependent coefficients, and also allows for the inclusion of moment conditions that will be used. For an overview and comparison of the literature on Yamada--Watanabe theory, we refer to the introduction of \cite{T}. 

The following theorem is our main result. 

  \begin{theorem}[Strong existence and uniqueness]
    \label{th:strong_existence_uniqueness} 
    Let $h>1$ and suppose Assumption~\ref{ass:ql}. Then, there exists $p_0\in(2,\infty)$ such that, for all $p\in(2,p_0]$, for all $q\in(ph,\infty)$: %
    if  $u_0\in L^p(\Om;B_{2,p,0}^{1-\nicefrac{2}{p}}(D))\cap L^q(\Om\times D)$, %
     then there exists a unique global solution $u$ to \eqref{eq:ql} in the sense of Definition \ref{def: strong sol}. Furthermore, $u$ has the following additional regularity for every $T\in (0,\infty)$:%
        \[
    u\in L^p(\Om;C([0,T];B_{2,p,0}^{1-\nicefrac{2}{p}}(D)))\cap L^q(\Om\times (0,T)\times D)\cap L^p(\Om\times(0,T);H_0^{1}(D)). %
        \]
    \end{theorem}

To make the statement of Theorem \ref{th:strong_existence_uniqueness} precise, we first detail the structural assumptions on \eqref{eq:ql}. The coefficients and nonlinearities are assumed to satisfy the following conditions for some $h>1$.  A comparison with the analogous assumption \cite[Asm.\ 6.1 ($h$)]{BV} for probabilistically weak existence is provided in Remark \ref{rem:assumptions} below.

  \begin{assumption}{QL($h$)}\label{ass:ql}\label{ass:weak existence} 
   The following hold:
       \begin{enumerate}[{\rm(i)}]
			\item\label{ass:2ndorder11} The coefficient functions $a_{ij}\colon (0,\infty) \times \Dom \times \R \to \R$ and $b_j \coloneqq (b_{n,j})_{n\geq 1} \colon (0,\infty) \times \Dom \times \R\to \ell^2$
            are measurable, uniformly bounded, Lipschitz continuous in the last component uniformly in $(t,x)$, and $\sup_{t,x,y} |\partial_{x_j} b_{n,j}(t,x,y)| < \infty$ for all $n$ and $j$.
			\item\label{ass:2ndorder12}
			One has $a_{ij} = a_{ji}$ and stochastic parabolicity holds, that is, there exists $\lambda>0$ such that, for all $t\in (0,\infty)$, $x\in D$, and $y\in \R$,
			$$
			\big(a_{ij}(t,x,y)-\frac12 \sum_{n \geq 1} b_{n,i}(t,x,y) b_{n,j}(t,x,y) \big)
			\xi_i \xi_j
			\geq  \lambda |\xi|^2, \quad \xi\in \R^{d}.
			$$
			\item\label{ass:2ndorder13} The nonlinearities $\Phi \colon (0,\infty) \times D \times \R \to \R^d$ and $\phi \colon (0,\infty) \times D \times \R \to \R$ are measurable, locally Lipschitz continuous in the last component  uniformly in $(t,x)$, 
            and satisfy the following growth condition: there exists a constant $C$ such that
			\begin{align}\label{eq:condPhiphi} \tag{G}
				|\Phi(t,x,y)| + |\phi(t,x,y)|&\leq C(1+|y|^{h}), \quad  t\in (0,\infty), x\in D,  y\in \R.
			\end{align}
			\item\label{ass:2ndorder14} The nonlinearity $\phi$ is dissipative in the following sense: there exists a constant $C$ such that for all $t\in (0,\infty)$, $x\in D$, $y,z\in \R$ there holds
			\begin{align}\label{eq:phi_dissip}
				(\phi(t,x,y) - \phi(t,x,z))(y-z) \leq C(y-z)^2.
			\end{align}
            The nonlinearity $\Phi$ %
            is dissipative in the following sense: there exists a constant $C$ such that for all $t\in (0,\infty)$, $x\in D$, $y,z\in \R$ there holds
            \begin{align}
            \label{eq:Phi_dissip}
                |\Phi(t,x,y) - \Phi(t,x,z)| \leq C\bigl( |y-z| + |y-z|^2 \bigr).
            \end{align}
			\item\label{ass:2ndorder15} The stochastic nonlinearity $g \coloneqq (g_n)_{n \geq 1} \colon (0,\infty) \times D \times \R \to \ell^2$ is measurable, Lipschitz continuous in the last component uniformly in $(t,x)$, and satisfies the following growth condition: there exists a constant $C$ such that
			\begin{align}
				\| g(t,x,y) \|_{\ell^2}&\leq C(1+|y|), \quad  t\in (0,\infty), x\in D,  y\in \R.
			\end{align}
		\end{enumerate}
     \end{assumption} 

    \begin{remark}[Comparison of the assumptions]\label{rem:assumptions}
    Our Assumption~\ref{ass:ql} is a slightly stronger version of~\cite[Asm.\ 6.1 ($h$)]{BV}.
    As was proved in ~\cite{BV},  the latter is sufficient for probabilistically weak existence, even for systems. However, the question of uniqueness is more delicate, and we can only treat the scalar case.   
    The differences between Assumption~\ref{ass:ql} and ~\cite[Asm.\ 6.1 ($h$)]{BV} in the scalar case are as follows:
    \begin{itemize}\itemsep-.45em 
        \item We assume $a_{ij}$  and $b_j$ to be  
        Lipschitz continuous in the last component instead of merely continuous.
        \item 
        In \cite{BV},  the dissipativity bound $\phi(t,x,y) y \leq \hat C (|y|^2+1)$ is required. This bound is implied by our dissipativity condition  \eqref{eq:phi_dissip} and Assumption~\ref{ass:ql}~\ref{ass:2ndorder13}.
    \end{itemize} 
The strengthenings above are necessary for the pathwise uniqueness proof in this paper, which relies on delicate approximations of the $L^1$-norm. On the other hand, the uniqueness proof also permits several relaxations of the assumptions used for weak existence in \cite{BV}: the equation may depend on $\Omega$, the local Lipschitz condition on $\phi$ is no longer required, and the growth of $g$ can be polynomial of order $h$ instead of linear. 
\end{remark}

\begin{remark}[Global solutions]\label{rem:loc t}
Since we consider global strong solutions, all coefficients and nonlinearities are defined on $(0,\infty)$ rather than on a bounded interval $(0,T)$. For ease of reading, we have chosen to make all constants in Assumption~\ref{ass:ql} uniform over $t\in(0,\infty)$. However, it is sufficient for the conditions on $g,\phi,$ and $\Phi$  to hold on every finite interval $[0,T]$ with a constant $C_T$. This sufficiency follows directly from Proposition \ref{prop:pathwise_uniqueness} (which is established on $[0,T]$) and the proof of Theorem \ref{th:strong_existence_uniqueness}. It also relies on the subtle point that the parameter $p_0$ from the weak existence proof in \cite{BV} is independent of the time $T$ (which follows by inspection of the proofs in~\cite{BV}) and the constants belonging to 
$g,\phi,$ and $\Phi$. However, $p_0$ does depend on the bounds and ellipticity constant belonging to $a_{ij}$ and $(b_{n,j})_{n \geq 1}$, so for these coefficients, we cannot localize the constants in time. 
\end{remark}

The notion of solution that we work with is as follows.  

\begin{definition}[Strong solutions]\label{def: strong sol} 
    Let  $h>1$ and suppose Assumption~\ref{ass:ql}. We call $u$ a \emph{global solution to \eqref{eq:ql}} if ${u}\colon \Omega\times [0,\infty)\to L^2(D)$ is $(\F_t)$-progressively measurable and satisfies  the following  for all $T \in (0,\infty)$: %
    \begin{enumerate}
    \item \label{it:sol1} One has $u \in L^2(\Omega \times (0,T); H^1_0(D))$ and  $u\in C([0,T];L^2(D))\cap  L^{2h}((0,T)\times D)$ a.s., 
    \item \label{it:sol2} the following identity holds a.s.\ in $C([0,T];H^{-1}(D))$:
    \[
    {u}(\cdot)= {u}_0+ \int_0^\cdot \partial_i (a_{ij}({u}) \partial_j {u}) + \partial_i \Phi_i({u}) +\phi({u}) \d s + \sum_{n\geq 1} \int_0^\cdot b_{n,j}({u}) \partial_j {u} + g_{n}({u})  \d {w}_n(s).
    \] 
    \end{enumerate}
   Likewise, if in the above we fix $T>0$ and $u\colon \Omega \times [0,T] \to L^2(D)$ is $(\F_t)$-progressively measurable and satisfies~\ref{it:sol1} and~\ref{it:sol2}, then we call $u$ a \emph{solution to \eqref{eq:ql} on $[0,T]$}.
\end{definition}

The following application illustrates Theorem \ref{th:strong_existence_uniqueness}. 

\begin{example}
	Consider the following quasilinear Allen--Cahn equation with gradient noise on a bounded, open set $D\subset\R^d$: 
	\begin{align*} 
		\left\{\quad
		\begin{aligned}
			\d u  &= \big[ \partial_i (a_{ij}(u) \partial_j u) +u-u^3\big] \d t
						+ \sum_{n\geq 1} b_{n,j}(u) \partial_j u \d w_n,\ \ &\text{ on }\Dom,\\
			u &= 0, \ \ &\text{ on }\partial\Dom,
			\\
			u(0)&=u_{0},\ \ &\text{ on }\Dom.
		\end{aligned}\right.
	\end{align*}
    Putting $\phi(t,x,u) \ceqq u - u^3$, Assumption~\ref{ass:ql}~\ref{ass:2ndorder13} and~\ref{ass:2ndorder14} are satisfied with $h=3$. Indeed, for the dissipativity, note that 
    $$(u-u^3-(v-v^3))(u-v)=(u-v)^2(1-(u^2+uv+v^2))\leq (u-v)^2.$$ 
    Thus, imposing Assumption~\ref{ass:ql}~\ref{ass:2ndorder11} and~\ref{ass:2ndorder12} on $a_{ij}$ and $(b_{n,j})_{n \geq 1}$, Theorem \ref{th:strong_existence_uniqueness} yields strong existence and uniqueness for the  Allen--Cahn equation above whenever $u_0\in L^p(\Om;B_{2,p,0}^{1-\nicefrac{2}{p}}(D))\cap L^q(\Om\times D)$ for  $p\in (2,\infty)$ small enough and $q\in (3p, \infty)$. 
\end{example}

Let us conclude this introduction by relating our results to the existing literature on similar quasilinear equations. A strong existence and uniqueness result such as Theorem \ref{th:weak_existence} on the torus can  be found in \cite[Th.\ 2.7]{DHV} for the case $b = \phi = 0$.
Their initial datum is assumed to belong to $L^p$ for every $p\in[1,\infty)$.
Using so-called kinetic solutions, they also treat the  
degenerate case (i.e.\ $\lambda=0$ in Assumption~\ref{ass:ql}~\ref{ass:2ndorder12}). We also mention \cite[\S5]{ASV}, which investigates equations on the torus $\mathbb{T}^d$ similar to \eqref{eq:ql} using a different approach. Both their initial datum $u_0$ and noise terms are more regular, but this way the solution $u$ turns out to be regular in time and space as well.

\subsubsection*{Acknowledgement} The authors thank Mark Veraar for encouraging them to work on the subject.

\section{Probabilistically weak solutions}

The proof of Theorem \ref{th:strong_existence_uniqueness} relies on a version of the Yamada--Watanabe theorem. 
Although the precise details of the latter are subtle and technical, it is  summarized  by: `Existence of a weak solution and pathwise uniqueness imply existence of strong solutions'. In this statement, weak and strong solutions are meant in a probabilistic sense. For details on Yamada--Watanabe theory for SPDEs with variational-like  solution notions, we refer to \cite{T}.  
Our notion of strong solutions was already given in Definition \ref{def: strong sol}. To discuss 
 weak solutions, we consider the following equation:

\begin{align}
		\label{eq:ql'}
		\tag{QLw}
		\left\{\quad
		\begin{aligned}
			\d \tilde{u}  &= \big[ \partial_i (a_{ij}(\tilde{u}) \partial_j \tilde{u}) + \partial_i \Phi_i(\tilde{u}) +\phi(\tilde{u}) \big] \d t
			\\  & \qquad \qquad
			+ \sum_{n\geq 1}  \big[ b_{n,j}(\tilde{u}) \partial_j \tilde{u} + g_{n}(\tilde{u}) \big] \d \tilde{w}_n,\ \ &\text{ on }\Dom,\\
			\tilde{u} &= 0, \ \ &\text{ on }\partial\Dom.%
		\end{aligned}\right.
	\end{align}

The notion of a weak solution to \eqref{eq:ql'} is as follows. 

\begin{definition}[Weak solution]\label{def: weak sol}
Let  $h>1$ and suppose Assumption~\ref{ass:ql}. For $T\in(0,\infty)$, we call $(\tilde{u},\tilde{W},(\tilde{\Omega},\tilde{\F},\tilde{\P},(\tilde{\F}_t)))$ a \emph{weak solution} to \eqref{eq:ql'}  on $[0,T]$ if $(\tilde{\Omega},\tilde{\F},\tilde{\P},(\tilde{\F}_t))$ is a filtered probability space, $\tilde{W}$ is an $\ell^2$-cylindrical Brownian motion on $(\tilde{\Omega},\tilde{\F},\tilde{\P},(\tilde{\F}_t))$, $\tilde{u}\colon \tilde{\Omega}\times [0,T]\to L^2(D)$ is $(\tilde\F_t)$-progressively measurable, $\tilde u \in L^2(\tilde\Omega \times (0,T); H^1_0(D))$, 
    $u\in C([0,T];L^2(D))\cap  L^{2h}((0,T)\times D)$ a.s., 
and a.s.,  for all $t\in[0,T]$ in $H^{-1}(D)$:
    \begin{equation}\label{eq: sol id in H^-1}
            \tilde{u}(t)= \tilde{u}(0)+ \int_0^t \partial_i (a_{ij}(\tilde{u}) \partial_j \tilde{u}) + \partial_i \Phi_i(\tilde{u}) +\phi(\tilde{u}) \d s + \sum_{n\geq 1} \int_0^t b_{n,j}(\tilde{u}) \partial_j \tilde{u} + g_{n}(\tilde{u})  \d \tilde{w}_n(s).
    \end{equation}
\end{definition}

The difference between Definitions \ref{def: strong sol} and \ref{def: weak sol} is that in the former, the probability space $(\Om,\F,\P)$, the cylindrical Brownian motion $W$ and the initial datum $u_0$ are fixed, while this is not the case in Definition \ref{def: weak sol}. Note that $u$ is a solution to \eqref{eq:ql} on $[0,T]$ if and only if $(u,W,(\Om,\F,\P,(\F_t)))$ is a weak solution to \eqref{eq:ql'} on $[0,T]$ and $u(0)=u_0$ a.s.

In this section, we discuss existence of probabilistically weak solutions for the Yamada--Watanabe approach. Pathwise uniqueness will be proved in Section~\ref{sec:uniqueness}. Eventually, we can conclude our main result, Theorem~\ref{th:strong_existence_uniqueness}, in the final Section~\ref{sec:conclusion}.

The next theorem was shown in~\cite[Th.\ 6.6]{BV}, taking Remarks~\ref{rem:assumptions} and~\ref{rem:loc t} into account. 

\begin{theorem}[Weak existence]\label{th:weak_existence} 
Let $h > 1$ and let Assumption~\ref{ass:ql} hold.
Then there exists a $p_0 \in (2,\infty)$ such that for every $p\in(2,p_0]$, $q\in(ph,\infty)$: if  $u_0\in L^p(\Om;B_{2,p,0}^{1-\nicefrac{2}{p}}(D))\cap L^q(\Om\times D)$, 
then for every $T\in (0,\infty)$, there exists  a weak solution  $(\tilde{u},\tilde{W},(\tilde{\Omega},\tilde{\F},\tilde{\P},(\tilde{\F}_t)))$ to \eqref{eq:ql'} on $[0,T]$ in the sense of Definition \ref{def: weak sol} with initial law $\tilde{\P}\#(\tilde{u}(0))=\P\#u_0$. In addition,  
        \[
    \tilde{u}\in L^p(\tilde{\Omega};C([0,T];B_{2,p,0}^{1-\nicefrac{2}{p}}(D)))\cap L^p(\tilde{\Omega}\times(0,T);H_0^{1}(D)) \cap L^q(\tilde{\Omega}\times (0,T)\times D).
        \]
\end{theorem}

\section{Pathwise uniqueness}
\label{sec:uniqueness}

Throughout this section,  we let $(\tilde{\Omega},\tilde{\F},\tilde{\P},(\tilde{\F}_t))$ be an arbitrary filtered probability space and $\tilde{W}$ an arbitrary $\ell^2$-cylindrical Brownian motion on $(\tilde{\Omega},\tilde{\F},\tilde{\P},(\tilde{\F}_t))$.

To show pathwise uniqueness, we refine the $L^1$-contraction approach from~\cite{HZ}. It uses a suitable approximation of the absolute value which we recall in the following lemma.

\begin{lemma}[Approximation of the absolute value]
    \label{lem:abs_value}
    For every $m \geq 1$ there exists a function $\psi_m \colon \R \to \R$ which is twice continuously differentiable, has bounded second derivative and satisfies the following properties  for all $\xi \in \R$:
    \begin{align}
    \psi_m'(0) &= 0, \label{it:av7} \\
        \psi_m(\xi) &\geq 0, \label{it:av1} \\
        \frac{\psi_m'(\xi)}{\xi} &\geq 0 \quad (\xi \neq 0), \label{it:av5} \\
        \psi_m''(\xi) &\geq 0, \label{it:av8} \\
        \psi_m(\xi) &\geq |\xi| - \e^{1-m} \label{it:av6} \\
        |\psi_m(\xi)| &\leq |\xi|, \label{it:av2} \\
        |\psi_m'(\xi)| &\leq 1, \label{it:av3} \\
        |\psi_m''(\xi)| &\leq \frac{2}{m|\xi|}. \label{it:av4}
    \end{align}
    In particular, for every $\xi \in \R$ we have $\psi_m(\xi) \to |\xi|$ as $m \to \infty$.
\end{lemma}

\begin{lemma}
	\label{lem:dissi}
    Let $C$ denote the dissipativity constant for the nonlinearity $\phi$ from Assumption~\ref{ass:weak existence}~\ref{ass:2ndorder14}. Let $\psi_m$ be as in  Lemma~\ref{lem:abs_value}.
	Then, for every $m \geq 1$, $t > 0$, $x \in D$ and $y,z \in \R$ we have
	\begin{align}
		\psi_m'(y-z) (\phi(t,x,y) - \phi(t,x,z)) \leq C (\psi_m(y-z) + \e^{1-m}).
	\end{align}
\end{lemma}

\begin{proof}
    Let $m \geq 1$, $t > 0$, $x \in D$ and $y,z \in \R$. We can assume $y \neq z$, as otherwise the claim is trivial by~\eqref{it:av1} and~\eqref{it:av7}. Then, using~\eqref{it:av5}, we have
    \begin{align}
        \psi_m'(y-z) (\phi(t,x,y) - \phi(t,x,z)) %
        &= \frac{|\psi_m'(y-z)|}{|y-z|} (\phi(t,x,y) - \phi(t,x,z)) (y-z),
    \end{align}
    and from~\eqref{it:av3} and Assumption~\ref{ass:weak existence}~\ref{ass:2ndorder14} we deduce that
	\begin{align}
		\psi_m'(y-z) (\phi(t,x,y) - \phi(t,x,z)) \leq \frac{1}{|y-z|} C (y-z)^2 = C |y-z|.
	\end{align}
	Now the claim follows from~\eqref{it:av6}.
\end{proof}

\begin{proposition}[Pathwise uniqueness]
    \label{prop:pathwise_uniqueness}
    Fix $h>1$ and let Assumption~\ref{ass:ql} be satisfied. 
    Also, let $T\in(0,\infty)$. 
    Suppose that $(\tilde u, \tilde W, \tilde \Omega, \tilde \F, \tilde \P, (\tilde \F_t))$ and $(\tilde v, \tilde W, \tilde \Omega, \tilde \F, \tilde \P, (\tilde \F_t))$ are weak solutions to~\eqref{eq:ql'} on $[0,T]$ in the sense of Definition~\ref{def: weak sol}, and suppose that  $\tilde u(0) = \tilde v(0)$ a.s. Then it holds that $\tilde u = \tilde v$ a.s.\ in $\Cont([0,T]; L^2(D))$.
\end{proposition}

\begin{proof}
		Let $t\in [0, T]$ be arbitrary.
		For every $m$ define the functional $\Psi_m \colon L^2(\Dom) \to \R$ by $\Psi_m(f) = \int_\Dom \psi_m(f) \d x$, where $\psi_m$ is as in Lemma~\ref{lem:abs_value}. By the properties of $\psi_m$ stated in Lemma~\ref{lem:abs_value}, the It\^o formula from~\cite[Th.\ 4.2, p.\ 65]{Pardoux} is applicable with the functional $\Psi_m$,  and applying it to the process $\tilde u-\tilde v$ yields the identity
		\begin{align}
			\begin{split}
			\label{eq:ito}
			&\int_\Dom \psi_m(\tilde u(t)-\tilde v(t)) \d x \\
            = -&\int_0^t \int_\Dom \psi_m''(\tilde u-\tilde v) (a_{ij}(\cdot,\tilde u) \partial_j \tilde u - a_{ij}(\cdot,\tilde v) \partial_j \tilde v) \partial_i (\tilde u-\tilde v) \d x \d s \\
			\quad -&\int_0^t \int_\Dom \psi_m''(\tilde u-\tilde v) (\Phi_i(\cdot,\tilde u) - \Phi_i(\cdot,\tilde v)) \partial_i (\tilde u-\tilde v) \d x \d s \\
			+&\int_0^t \int_\Dom \psi_m'(\tilde u-\tilde v)(\phi(\cdot,\tilde u)-\phi(\cdot,\tilde v)) \d x \d s \\
			+&\sum_n \int_0^t \int_\Dom \psi_m'(\tilde u-\tilde v)[b_{n,j}(\cdot,\tilde u) \partial_j \tilde u - b_{n,j}(\cdot,\tilde v) \partial_j \tilde v + g_n(\cdot,\tilde u) - g_n(\cdot,\tilde v)] \d x \d w_n(s) \\
			+&\frac{1}{2} \int_0^t \int_\Dom \psi_m''(\tilde u-\tilde v) \sum_n \Bigl| \sum_j (b_{n,j}(\cdot,\tilde u) \partial_j \tilde u - b_{n,j}(\cdot,\tilde v) \partial_j \tilde v) + g_n(\cdot,\tilde u) - g_n(\cdot,\tilde v) \Bigr|^2 \d x \d s \\
			= -&\I - \II + \III + \IV + \VV.
			\end{split}
		\end{align}
        Note that we suppress certain variable dependencies to increase readability. 
        
        To handle term $\VV$, we expand $$b_{n,j}(\cdot,\tilde u) \partial_j \tilde u - b_{n,j}(\cdot,\tilde v) \partial_j \tilde v = b_{n,j}(\cdot,\tilde u) \partial_j (\tilde u-\tilde v) + (b_{n,j}(\cdot,\tilde u) - b_{n,j}(\cdot,\tilde v)) \partial_j \tilde v.$$ Then, for $\eps>0$, we apply the inequality $\nicefrac{1}{2} |y+z|^2 \leq (\nicefrac{1}{2} + \eps) y^2 + C_\eps z^2$ to obtain
		\begin{align}
			\label{eq:b_absorb}
			\VV &\leq (\nicefrac{1}{2} + \eps) \int_0^t \int_\Dom \psi_m''(\tilde u-\tilde v) \sum_n \Bigl|\sum_j b_{n,j}(\cdot,\tilde u) \partial_j (\tilde u-\tilde v)\Bigr|^2 \d x \d s \\
			&\quad+ C_\eps \int_0^t \int_\Dom \psi_m''(\tilde u-\tilde v) \sum_n \Bigl| \sum_j (b_{n,j}(\cdot,\tilde u) - b_{n,j}(\cdot,\tilde v)) \partial_j \tilde v + g_n(\cdot,\tilde u) - g_n(\cdot,\tilde v) \Bigr|^2 \d x \d s.
		\end{align}
		On the other hand, rewrite with a similar expansion
		\begin{align}
			\label{eq:a_absorb}
			\begin{split}
			\I ={} &\int_0^t \int_\Dom \psi_m''(\tilde u-\tilde v) a_{ij}(\cdot,\tilde u) \partial_j (\tilde u-\tilde v) \partial_i (\tilde u-\tilde v) \d x \d s \\
			+ &\int_0^t \int_\Dom \psi_m''(\tilde u-\tilde v) (a_{ij}(\cdot,\tilde u) - a_{ij}(\cdot,\tilde v)) \partial_j \tilde v \partial_i (\tilde u-\tilde v) \d x \d s.
			\end{split}
		\end{align}
        Now, keeping~\eqref{it:av8} in mind, we use Assumption~\ref{ass:weak existence}~\ref{ass:2ndorder11} and~\ref{ass:2ndorder12} pointwise in $s$ and $x$ to handle the first terms on the right-hand sides of~\eqref{eq:b_absorb} and~\eqref{eq:a_absorb}. More precisely, since $b$ is uniformly bounded, we can take $\eps$ small enough (depending on $b$ and the ellipticity constant $\lambda$) such that 
        $\eps\sum_n  |\sum_j b_{n,j}(\cdot,\tilde u) \partial_j (\tilde u-\tilde v) |^2\leq \lambda|\nabla(\tilde u-\tilde v)|^2$, so that the sum of the first terms on the right-hand sides of~\eqref{eq:b_absorb} and~\eqref{eq:a_absorb} is nonpositive. %

		Next, we treat term $\III$. By  Lemma~\ref{lem:dissi} we readily find
		\begin{align}
			\III \leq C \int_0^t \int_\Dom \psi_m(\tilde u-\tilde v) \d x \d s + CT|D| \e^{1-m}.
		\end{align}
		Plugging all back into~\eqref{eq:ito} and taking expectations, we deduce
		\begin{align}
			&\E \int_\Dom \psi_m(\tilde u(t)-\tilde v(t)) \d x \\
            \lesssim{} &\E \int_0^t \int_\Dom \psi_m''(\tilde u-\tilde v) |(a_{ij}(\cdot,\tilde u) - a_{ij}(\cdot,\tilde v)) \partial_j \tilde v \partial_i (\tilde u-\tilde v)| \d x \d s \\
            +\, &\E \int_0^t \int_\Dom \psi_m''(\tilde u-\tilde v) |(\Phi_i(\cdot,\tilde u) - \Phi_i(\cdot,\tilde v)) \partial_i (\tilde u-\tilde v)| \d x \d s \\
            +\, &\E \int_0^t \int_\Dom \psi_m(\tilde u-\tilde v) \d x \d s + CT|D| \e^{1-m} \\
			+\, &\E \int_0^t \int_\Dom \psi_m''(\tilde u-\tilde v) \sum_n \Bigl| \sum_j (b_{n,j}(\cdot,\tilde u) - b_{n,j}(\cdot,\tilde v)) \partial_j \tilde v + g_n(\cdot,\tilde u) - g_n(\cdot,\tilde v) \Bigr|^2 \d x \d s.
		\end{align}
        Note that the term $\IV$ disappeared due to its martingale property.
		Using the bound~\eqref{it:av4}, Lipschitz continuity of $a$,  Assumption~\ref{ass:ql}~\ref{ass:2ndorder14},  Lipschitz continuity and  boundedness of $b$, Lipschitz continuity of $g$, and Young's inequality, we further get
		\begin{align}
			\E \int_\Dom \psi_m(\tilde u(t)-\tilde v(t)) \d x &\lesssim \frac{1}{m} \E \int_0^t \int_\Dom \bigl( |\partial_j \tilde v|+1 + |\tilde u-\tilde v| \bigr)|\partial_j (\tilde u-\tilde v)| \d x \d s \\
			&\quad+ \E \int_0^t \int_\Dom \psi_m(\tilde u-\tilde v) \d x \d s + CT|D| \e^{1-m} \\
            &\quad+ \frac{1}{m} \E \int_0^t \int_\Dom |\partial_j \tilde v|^2 + |\tilde u-\tilde v| \d x \d s \\
			&\lesssim \frac{1}{m} \E \int_0^T \int_\Dom 1 + |\nabla \tilde u|^2 + |\nabla \tilde v|^2 + |\tilde u|^2 + |\tilde v|^2 \d x \d s \\
			&\quad+ \E \int_0^t \int_\Dom \psi_m(\tilde u-\tilde v) \d x \d s + CT|D| \e^{1-m}.
		\end{align}
		Gronwall's inequality now gives
		\begin{align}
			\E \int_\Dom \psi_m(\tilde u(t)-\tilde v(t)) \d x \lesssim \frac{1}{m} \Bigl(1+\| \tilde u \|_{L^2(\Omega \times (0,T); H^1(\Dom))}^2 + \| \tilde v \|_{L^2(\Omega \times (0,T); H^1(\Dom))}^2 \Bigr) + \e^{1-m}.
		\end{align}
		Recall from Lemma~\ref{lem:abs_value} that $0 \leq \psi_m(x) \to |x|$ as $m\to \infty$.
		Since the norms on the right-hand side are finite, Fatou's lemma allows us to take the limit $m\to \infty$,   leading to 
		\begin{align}
			\E \int_D |\tilde u(t)-\tilde v(t)| \d x \leq 0.
		\end{align}
		Thus, almost surely we have $\tilde u(t,x)=\tilde v(t,x)$ for almost every $x\in D$. In particular, a.s.\ $\tilde u(t) = \tilde v(t)$ in $L^2(D)$. Since $t$ was arbitrary, this gives the claim.
	\end{proof}

\section{Conclusion of the main result}
\label{sec:conclusion}

\begin{proof}[Proof of Theorem \ref{th:strong_existence_uniqueness}]
We have established weak existence in Theorem \ref{th:weak_existence}  and pathwise uniqueness in Proposition \ref{prop:pathwise_uniqueness} on all  intervals $[0,T]$ with $T\in(0,\infty)$. Therefore,  we can invoke the Yamada--Watanabe theorem to obtain strong existence on each $[0,T]$, and afterwards deduce global strong existence. However, the solution notion to which we have to apply the Yamada--Watanabe theorem is non-standard, so we resort to the  general Yamada--Watanabe result of \cite[Th.\ 3.1]{T}. 

Fix $T \in (0,\infty)$. Let us specify how we apply this result using the notation from \cite{T}.  
For the flexible $\mathrm{C}$-weak solution notion of \cite[Def.\ 2.5]{T}, we use  $Z\ceqq H^{-1}(D)$,  
$$
\mathbb{B}\ceqq C([0,T];B_{2,p,0}^{1-\nicefrac{2}{p}}(D))\cap  L^{q}((0,T)\times D)\cap L^p((0,T);H_0^{1}(D))
$$ 
and let the collection of  solution constraints be given by $\mathrm{C}\ceqq\{(1),(4),(6)\}$ as defined in \cite[p.\ 11, p.\ 12]{T}. In condition (6) of \cite[p.\ 12]{T},  we use   $f_p,f_q\col\mathbb{B}\to \R$ defined by $f_p(u)\ceqq \|u\|_{L^p((0,T);H_0^{1}(D))}+\|u\|_{C([0,T];B_{2,p,0}^{1-\nicefrac{2}{p}}(D))}$ and $f_q(u)\ceqq \|u\|_{L^{q}((0,T)\times D)}$ and require that $f_\alpha(u)\in L^\alpha(\Omega)$ for $\alpha\in\{p,q\}$. Each $f_\alpha$ is continuous, hence Borel measurable. In \cite[p.\ 12]{T},  condition (6) is formulated for a single map $f$, but it can be extended to arbitrarily many maps $f_\alpha$, by redefining   
$
\Gamma_{(6)}\ceqq \cap_\alpha \{\tilde{\mu}\in\mathcal{P}(S_1\times S_2):\int_{\mathbb{B}}|f_\alpha|^\alpha \d (\Pi_1\#\tilde{\mu}) \}  
$  
in \cite[(3.26)]{T}. The terms in  our equation satisfy the measurability conditions of  \cite[(2.3), (2.4)]{T} thanks to Assumption~\ref{ass:ql} and \cite[Lem.\ 2.2]{T}.

With the choices above, \cite[Asm.\ 2.1]{T} is satisfied. %
Moreover, if $u_0\in L^p(\Om;B_{2,p,0}^{1-\nicefrac{2}{p}}(D))\cap L^q(\Om\times D) $ is given on a probability space $(\Om,\F,\P)$, then we can define $\mu\ceqq \P\#u_0$ (the law of $u_0$) and the following are equivalent: 
\begin{enumerate}
    \item $(\tilde u, \tilde{W},(\tilde \Omega, \tilde \F, \tilde \P, (\tilde \F_t)))$ is a $\mathrm{C}$-weak solution in the sense of \cite[Def.\ 2.5]{T} with  \\$\tilde{\P}\# (\tilde{u}(0))=\mu$, 
    \item $(\tilde u, \tilde{W},(\tilde \Omega, \tilde \F, \tilde \P, (\tilde \F_t)))$ is a weak solution on $[0,T]$ in the sense of Definition \ref{def: weak sol} with $\tilde{\P}\# (\tilde{u}(0))=\mu$, and additionally, $$\tilde u\in L^p(\tilde\Om;C([0,T];B_{2,p,0}^{1-\nicefrac{2}{p}}(D)))\cap L^q(\tilde\Om\times (0,T)\times D)\cap L^p(\tilde\Om\times(0,T);H_0^{1}(D)).$$ %
\end{enumerate} 
For the equivalence, we take into account \cite[Rem.\ 2.10]{T} and note that $p\in(2,\infty)$ and  $q>ph>2h$, so $\mathbb{B}\hookrightarrow C([0,T];L^2(D))\cap  L^{2h}((0,T)\times D)$. 
In particular, if (i) holds, then $\tilde u\in\mathbb{B}$ a.s.\ \cite[Def.\ 2.5]{T}, and thus   
    $\tilde u\in C([0,T];L^2(D))\cap  L^{2h}((0,T)\times D)$ a.s., as required in Definition \ref{def: weak sol}. Moreover, the integrability with respect to $\tilde \Om$ in (ii) is encoded by our choice of $f_p$ and $f_q$ in condition (6) that was included in the set of constraints $\mathrm{C}$. 

Since (ii)$\Rightarrow$(i), existence of a $\mathrm{C}$-weak solution is provided by Theorem \ref{th:weak_existence}. Furthermore, pathwise uniqueness of $\mathrm{C}$-weak solutions is provided by Proposition \ref{prop:pathwise_uniqueness}, since   any $\mathrm{C}$-weak solution is in particular a weak solution in the sense of Definition \ref{def: weak sol} by (i)$\Rightarrow$(ii). 
An application of \cite[Th.\ 3.1 (a)$\Rightarrow$(c)]{T} now yields all the claims of Theorem \ref{th:strong_existence_uniqueness}, including the additional regularity of $u$, on the finite time interval $[0,T]$.  %
The remaining task is to extend this to a unique global solution.

Observe that if $u$ is a global solution to \eqref{eq:ql}, then $u|_{[0,T]}$ is a solution to \eqref{eq:ql} on $[0,T]$ for each $T\in(0,\infty)$. Thus, uniqueness of global solutions follows from the uniqueness of solutions on $[0,T]$. To obtain existence of a global strong solution, we apply the first part of the proof on $[0,n]$ for $n \in \N$, yielding countably many strong solutions $u^n$ to~\eqref{eq:ql} on $[0,n]$. Pathwise uniqueness yields $u^m|_{[0,n]} = u^n$ a.s. for $m \geq n$. Hence, patching together the solutions $(u^n)_n$ gives rise to a global strong solution $u$ to~\eqref{eq:ql} with the claimed regularity. 
\end{proof}

\end{document}